\documentclass{article}
\usepackage[english]{babel}
\usepackage[utf8]{inputenc}
\usepackage[T1]{fontenc}
\usepackage{amsmath}
\usepackage{amsfonts}
\usepackage{amsthm}
\usepackage{amssymb}
\usepackage{graphicx}
\usepackage{dsfont}

\newtheorem{theorem}{Theorem}
\newtheorem{corol}[theorem]{Corolary}
\newtheorem{definition}[theorem]{Definition}
\newtheorem{lemma}[theorem]{Lemma}
\newtheorem{nota}[theorem]{Remark}
\newtheorem{prop}[theorem]{Proposition}

\title{Viscosity solution of a Delta Greek nonlinear Black-Scholes equation}

\author{\textit{Rui M.P. Almeida\footnote{ralmeida@ubi.pt}, Te\'{o}filo D. Chihaluca\footnote{teofilo.chihaluca@ubi.pt} and Jos\'{e} C.M. Duque\footnote{jduque@ubi.pt}}\\ University of Beira Interior\\ Center of Mathematics and Applications\\ Covilh\~a, Portugal\\ }
  \date{\today}
\begin{document}

\maketitle




\begin{abstract}
In this paper, a class of nonlinear option pricing models involving transaction costs is considered.
The diffusion coefficient of the nonlinear parabolic equation for the price $V$ is assumed to be a linear function
of the option's underlying asset price and the Gamma Greek $V_{xx}$. The main aim of this work is to study the governing PDE of the  Delta Greek. The existence of viscosity solutions is proved using the vanishing viscosity method. Regularizing the equation by adding a small perturbation to the initial problem, a sequence of approximate solutions $u^{\varepsilon}$ is constructed and then the method of weak limits is applied to prove the convergence of the sequence to the viscosity solution of the Delta equation. The approximate problems constructed are shown to have good regularity, which allows the use of efficient and robust numerical methods.\\
\end{abstract}
\textbf{Keywords}
Delta equation; transaction costs; convergence; vanishing viscosity method.

\section{Introduction}

In financial mathematics, the Black-Scholes model is
frequently used for pricing derivatives by means of the reversed-time
parabolic partial differential equation (\ref{eq:1})(\cite{Mathias2009,vsevcovic2011analytical,wilmott1999}).
\begin{equation}  \label{eq:1}
0=V_{\tau}+\frac{1}{2}\sigma^{2}S^{2}V_{SS}+(r-q)SV_{S}-rV,\qquad S>0,\qquad
\tau\in]0,T[.
\end{equation}
In (\ref{eq:1}) $V$ is the option value, $S$ the underlying asset price, $%
\tau$ the time, $T$ the expiry date, $\sigma$ the volatility, $r$ the
riskless interest rate and $q$ the dividend rate. Equation (\ref{eq:1}) has
an analytical solution (see \cite{FMK}) and to obtain approximate solutions
there are also several reliable numerical methods, such as the binomial
method, the Monte Carlo method, finite difference and finite element
methods. For further details we refer the reader to the survey book \cite%
{Kwok_book}.

The Black-Scholes equation (\ref{eq:1}) is very effective in an idealistic
market, for example, without transaction costs. However transaction costs may
arise, for example, when trading securities. Although they are small in
general, they can lead to an increase in the option price, in which case the
Black-Scholes pricing methodology is no longer valid since perfect hedging
is impossible. Consequently, different models have been proposed to modify
equation (\ref{eq:1}) in order to accommodate transaction costs, such as
those in \cite{HEL,Avellada,Barles1998,MKR}. In these models, the constant
volatility is replaced by a modified volatility function which can depend on
time, on the asset price, on the option value and its derivatives. The
resulting model is a nonlinear equation in nondivergence form.

For the general nonlinear Black-Scholes equation, an explicit solution is
unknown and the numerical techniques available are far less than for the
linear model. Barles and Soner proved in \cite{Barles1998} the existence of
viscosity solutions for their model and made some simulations using finite
differences but no convergence analysis was performed. However, it is known
that explicit schemes have the disadvantage that restrictive conditions on
the discretization parameters (for instance, the ratio of the time and space
step) are needed to obtain stable, convergent schemes. Moreover, the
convergence order is only one in time and two in space. Pooley et all. \cite%
{pooley2003numerical} numerically studied the convergence of some finite
difference schemes applied to a nonlinear Black-Scholes equation and
presented some examples where non-monotone discretization schemes (such as
standard Crank-Nicolson time stepping) can converge to incorrect
solutions, or lead to instability. In \cite{DFJ04}, the authors combine
high-order compact difference scheme techniques to construct numerical
solutions of the transformed non linear equation using the transformation $%
x=ln(S/K)$ with frozen values of the nonlinear volatility term to make the
formulation linear. This transformation transforms the spatial domain $%
[0,\infty]$ to $[-\infty, \infty]$ and, in computations, this infinite
domain has to be truncated, which essentially omits the degeneracy of the
equation at $S=0$. They show that the finite difference solution converges
locally uniformly to the unique viscosity solution of the equation.
Ankudinova and Ehrhardt \cite{AKA} used a Crank-Nicolson method combined
with a high order compact difference scheme to construct a numerical scheme
for the linearized Black-Scholes equation. Company, Jodar and
Pintos in \cite{CJP09} proposed a semi-discretization technique which
approximates the nonlinear equation with a system of ordinary differential
equations and solved the system using the backward Euler scheme. In order to make
the high order scheme work, a smoothing technique for the payoff condition
is used, which essentially changes the nature of the pricing problem. In
\cite{LW13}, a method based on an upwind finite difference
scheme for the spatial discretization and on a fully implicit time-stepping
scheme is developed. The authors prove that the approximate solution converges
unconditionally to the viscosity solution of the equation.

The partial derivatives of the solution, the Greeks, are considered of major importance in finance,see for example \cite{vsevcovic2011analytical}. In particular,
the first spatial derivative, referred to as the Delta Greek in finance, is
the key for the hedging process in time, that is portfolio projection against
market movements as it follows from the Black-Scholes hedging arguments.
There are a few papers concerning the calculation of the Delta Greek
directly. Radoslav Valkov in \cite{Valkov2015} presented a convergence
analysis of a positivity-preserving fitted finite volume element method
(FVEM) for a generalized Black-Scholes equation transformed on finite
interval, degenerating on both boundary points. He first formulated the FVEM
as a Petrov-Galerkin finite element method using a spatial discretization,
previously proposed by the author in \cite{Valkov51}. The Garding coercivity
of the corresponding discrete bilinear form was established. He obtained
stability and error bounds for the solution of the fully-discrete system.
Analysis of the impact of the finite domain transformation on the numerical
solution of the original problem was given. Recently, Koleva and Vulkov (\cite%
{KV18}) constructed and analysed monotone and sign-preserving finite difference
schemes for the Delta equation. They proposed some Newton and Picard
iterative procedures for solving the non-linear systems of algebraic equations.

The theory of continuous viscosity solutions for fully nonlinear second-order elliptic and parabolic equations has been introduced by Crandall-Lions for the Hamilton-Jacobi equations and is well established nowadays. For mores details on this theory we refer the reader to \cite{Crandall1992} and references therein. However, the Delta equations has discontinuous initial condition, so isn't to expect
continuous solution and the conventional theories of viscosity solutions do
not apply. There are several definitions of discontinuous viscosity
solutions, but most of then are rather ad hoc. In \cite{GB03b} we can see a
survey of the latest development on the uniqueness and regularity of the
discontinuous solutions of the Hamilton-Jacobi equation, with discontinuous
initial data that are continuous outside a set of measure zero. They proved
that the discontinuous solutions of the problem is unique when the initial
condition is everywhere continuous. They also
clarified the connections among the discontinuous solutions from different
notions.

Usually, a high order method requires that the solution to the PDE is
sufficiently smooth in order to achieve the expected order of convergence.
However, it is known that the non-linear Black-Scholes equation generally does not have classic smooth solutions, but only viscosity solutions.
Therefore, a numerical solution to the nonlinear Black-Scholes equation by a
high order numerical scheme is not necessarily more accurate than that from
a first-order discretization scheme, mainly due to the non-smoothness of the
given data and the exact solution. Due to the Greeks being relevant for the
quantitative analysis, reliable numerical methods are required for the
pricing of options which not only provide a good approximation for the
price, but also for its derivatives. In this work we study a
simplification of the nonlinear Black-Scholes equation proposed by Barles
and Soner (\cite{Barles1998}) given focus to the corresponding Delta
equation. We prove the existence and uniqueness of possibly discontinuous viscosity
solutions for the Delta equation. We first regularize the equation by adding
a small perturbation parameter and then apply the method of weak limits to
prove the convergence of the classical solutions of the regularized problem
to the viscosity discontinuous solution of the Delta equation. The main goal
of this work is to create a basis to establish the convergence of high order
robust numerical methods since the regularized problem has good smooth
solutions.

The remainder of this paper is organized as follows. In Section $2$, the
problem is described. In Section $3$, we define the regularized approximate
problem as a general nonlinear equation. Section $4$ is dedicated to obtain
some à priori estimatives. In Section $5$ we prove the convergence and,
finally, in Section $5$, we draw some comments.

\section{The nonlinear model}

In 1998 Barles and Soner \cite{Barles1998} developed a complex model that
modifies Equation (\ref{eq:1}) and accommodates transaction costs. Following
the Hedges and Neuberger utility function approach in \cite{HAN}, they
proposed the volatility function
\begin{equation}  \label{eq:546}
\tilde \sigma=\sigma^{2} \left( 1+\Psi(
e^{r(T-\tau)}a^{2}S^{2}V_{SS})\right) ,
\end{equation}
where $\sigma $ is the historical volatility, $a=\frac{k}{\sqrt{\epsilon}}$
and the function $\Psi (A)$ is the solution to the following nonlinear
differential equation (ODE)
\begin{equation}  \label{1982}
\Psi^{\prime }(A)=\frac{\Psi (A)+1}{2 \sqrt{A\Psi(A)}-A}, \quad A \neq 0,
\end{equation}
with the initial condition
\begin{equation}  \label{1982a}
\Psi(0)=0.
\end{equation}

In this way Equation (\ref{eq:1}) becomes the following non-linear
Black-Scholes equation with Barles and Soner's model,

\begin{equation}  \label{eq:princ}
0=V_{\tau}+\frac{1}{2}\sigma^{2}\Big( 1+ \Psi(e^{r(T-\tau)}a^{2}S^{2}V_{SS})%
\Big)S^{2}V_{SS}+(r-q)SV_{S}-rV, \quad S>0,\, \tau \in]0,T[,
\end{equation}
which we will consider for European options, where $S$ pays out continuous
dividend $qSdt$ with time step $dt$.

A European call option allows the buyer to buy an asset of value $S$ for a
value $K$ on maturity date $T$, while an European put option allows the
holder to sell an asset of value $S$ for a value $K$ on maturity date $T$.
For the sake of simplicity, we will only consider the call option. Since the
option can only be exercised on maturity, we complement Equation (\ref%
{eq:princ}) with the following conditions, in order to avoid arbitrariness:
\begin{eqnarray}
V(S,T)=\max \{ S - K,0 \},&\text{when}&S \geq 0;\label{eq:461}\\
\lim_{S \to\infty} \frac{V(S,\tau)}{S-Ke^{-r(T-\tau)}}=1 ,&\text{for}&\tau \in [0,T];\label{eq:462}\\
V(0,\tau)=0,&\text{for}&\tau \in [0,T];\label{eq:550}\\
\lim_{S \to\infty} V_{S}(S,\tau)=1,&\text{for}&\tau \in [0,T].\label{eq:463}
\end{eqnarray}

Barles and Soner proved the existence of a viscosity solution for the
European option with volatility given by (\ref{eq:546}). Their numerical
results indicate an economically significant price difference between the
standard Black-Scholes model and the non-linear model with transaction
costs. 

Although an explicit solution for Problem (\ref{1982})-(\ref{1982a}),
is unknown, the following Lemmas provide some useful information about its
behaviour.

\begin{lemma}[\cite{Barles1998}] A straightforward analysis of the ordinary differential
equation (\ref{1982}) reveals that
\begin{equation}  \label{eq:t}
\lim_{A \to\infty}\frac{\Psi(A)}{A}=1 \qquad \text{and} \qquad \lim_{A
\to-\infty}\Psi(A)=-1.
\end{equation}
\end{lemma}

\begin{lemma}[\cite{Barles1998}]
If
\begin{equation*}
A\geq 0 \qquad \text{then} \qquad \Psi(A) \geq 0
\end{equation*}
and if
\begin{equation*}
A\leq 0 \qquad \text{then} \qquad -1< \Psi(A) \leq 0.
\end{equation*}
\end{lemma}

\begin{lemma}[\cite{Barles1998}] The differential $\Psi^{\prime }(A)\geq 0$.
\end{lemma}


\begin{lemma}[\cite{CJP09}] If $\ \ A\geq 0\qquad $vthen $\ \Psi (A)\leq C_{1}A+C_{2}$ \
with $\ \ C_{1}\approx 1.1$ \ and $\ C_{2}\approx 2.62$.
\end{lemma}

\begin{figure}[!htb]
\centering
\includegraphics[width=0.45\textwidth]{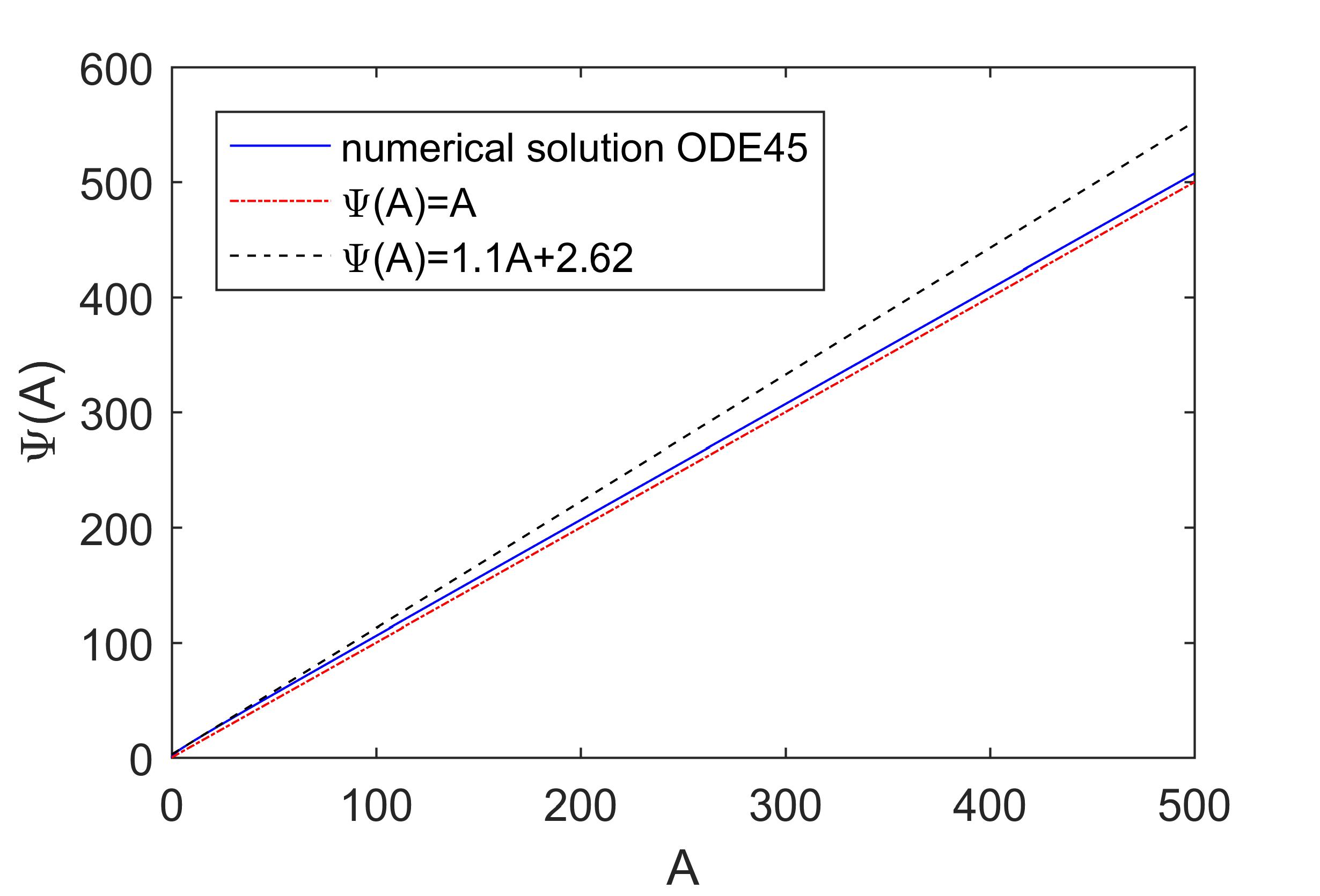}
\caption{Numerical simulation of $\Psi$.}
\label{fig1}
\end{figure}
A numerical solution for Problem (\ref{1982})-(\ref{1982a}) is represented
in Figure \ref{fig1}. By the lemmas above and figure \ref{fig1} is
reasonable to simplify the model considering $\Psi(A)=A$ for $A>0$. So, in
this work, we study the Black-Scholes Equation (\ref{eq:princ}), in which
the volatility is assumed to be a function of the underlying asset $S$, and
time $\tau$ and the Gamma of option (the Greek Gamma is a second derivative $%
V_{SS}$) that is,
\begin{equation}  \label{eq:400}
0=V_{\tau}+\frac{1}{2}\sigma^{2}\left( 1+
e^{r(T-\tau)}a^{2}S^{2}V_{SS}\right)S^{2}V_{SS}+(r-q)SV_{S}-rV, \quad S>0,\,
\tau \in ]0,T[.
\end{equation}

It is known that $V_{SS}\geq 0$ for European Puts and Calls in the absence
of transaction costs. In case $V_{SS}<0$, Problem (\ref{eq:400}) is
ill-posed and without solution for general pay-off functions (see \cite%
{HEL,Avellada} for similar cases). Therefore we will henceforth assume that $%
V_{SS}>0$.

\subsection{Delta equation}

Differentiating (\ref{eq:400}) with respect to $S$, the resulting equation
is
\begin{equation}  \label{eq:4000}
V_{S\tau}+\left(\frac{1}{2}\sigma^{2}\Big( 1+ e^{r(T-\tau)}a^{2}S^{2}V_{SS}%
\Big)S^{2}V_{SS}\right)_{S}+(r-q)SV_{SS}-qV_S=0.
\end{equation}
Differentiating (\ref{eq:461}), we have
\begin{eqnarray}  \label{eq:40010}
V_{S}(S,T)=%
\begin{cases}
0, \qquad S<K \\
1,\qquad S>K \\
\end{cases}%
.
\end{eqnarray}
Considering $x=S$, $t=T-\tau$ and $u(x,t)=V_S(S,T-\tau)$, we obtain the
parabolic partial differential equation in divergence form,
\begin{equation}  \label{delta_eq}
u_t=\left( \frac{1}{2}\sigma^{2}\big( 1+ e^{rt}a^{2}x^{2}u_{x}\big)%
x^{2}u_{x}\right)_{x}+(r-q)xu_x-qu, \quad x>0,\quad t \in]0,T].
\end{equation}
which we will call Delta equation. We will consider the corresponding
initial condition
\begin{eqnarray}  \label{eq:4001}
u(x,0)=u_0(x)=%
\begin{cases}
0, \qquad 0< x<K \\
0.5, \qquad x=K \\
1,\qquad x>K \\
\end{cases}%
\end{eqnarray}
By (\ref{eq:463}), the boundary conditions are
\begin{eqnarray}  \label{eq:4002}
u(0,t)=g_1(t)=0,\quad t\in[0,T]
\end{eqnarray}
and
\begin{eqnarray}  \label{eq:4002a}
\lim_{x\to \infty}u(x,t)=1,\quad t\in[0,T].
\end{eqnarray}
Note that condition $u(0,t)=0 $ is necessary in order to guarantee
compatibility between the equation and the initial data. It is not possible
to define $u_0$ in $x=K$ by the Black-Scholes model, since $V(S,T)$
is not differentiable at $S=K$. However, as we shall see below, this value
is not essential, so we will consider the value $u_0(K)=0.5$ for practical
purposes.

\begin{prop}
If function $V(S,\tau)$ is a solution of Problem (\ref{eq:princ})-(\ref%
{eq:463}), then $u(x,t)=V_S(S,T-\tau)$ is a solution of problem (\ref%
{delta_eq})-(\ref{eq:4002a}). Conversely if $u(x,t)$ is a solution of
Problem (\ref{delta_eq})-(\ref{eq:4002a}), then
\begin{equation}  \label{tans_inv}
V(S,\tau)=\int_0^Su(x,T-\tau)\ dx
\end{equation}
is a solution of problem (\ref{eq:princ})-(\ref{eq:463}).
\end{prop}

\begin{proof}
The first statement was proved above.
Suppose now that $V(x,\tau)$ is given by (\ref{tans_inv}). Then making the change of variable $t=T-\tau$, it follows that
$$V_{\tau}(S,\tau)=\int_0^Su_{\tau}(x,T-\tau)\ dx=\int_0^S-u_{t}(x,t)\ dx=$$
$$=-\int_0^S\left( \frac{1}{2}\sigma^{2}\big( 1+ e^{rt}a^{2}x^{2}u_{x}\big)x^{2}u_{x}\right)_{x}+(r-q)xu_x-qu\ dx.$$
Using the integration by parts, we obtain
$$V_{\tau}(S,\tau)=-\frac{1}{2}\sigma^{2}\big( 1+ e^{rt}a^{2}S^{2}u_{S}(S,t)\big)S^{2}u_{S}(S,t)-(r-q)Su(S,t)+q\int_0^Su(x,t)\ dx$$
Differentiating (\ref{tans_inv}) we have $V_S(S,\tau)=u(S,T-\tau)$ and $V_{SS}(S,\tau)=u_S(S,T-\tau)$. Substituting in the last equation, we obtain (\ref{eq:princ}).
In order to obtain (\ref{eq:461}), we just need to calculate de integral\\
$$V(S,T)=\int_0^Su(x,0)\ dx=\left\{\begin{array}{lll}
0&,& S<K\\
0&,& S=K\\
S-K&, & S>K\end{array}\right..$$
As we mentioned earlier any real value of $u(K,0)$ verifies (\ref{eq:461}). Condition (\ref{eq:462}) is obtained applying the L'Hôpital's rule, that is,
$$\lim_{S \to\infty} \frac{V(S,\tau)}{S-Ke^{-r(T-\tau)}}=\lim_{S \to\infty}\frac{\int_0^Su(x,T-\tau)\ dx}{S-Ke^{-r(T-\tau)}}=\lim_{S \to\infty}\frac{u(S,T-\tau)}{1}=1$$
 by (\ref{eq:4002a}). Condition (\ref{eq:550}) is obvious. Finally, condition (\ref{eq:463}) is the same as condition (\ref{eq:4002a}).
\end{proof}

Since it is difficult to deal with an infinite spatial domain, we will
consider the spatial domain $]0,b[$ with $b$ sufficiently far from $K$ and
substitute equation (\ref{eq:4002a}) by
\begin{equation}
u(b,t)=1,\quad t\in \lbrack 0,T].  \label{eq_fr_b}
\end{equation}%
There are two main problems  concerning the Delta function that we need to deal with.
First, the initial condition is not smooth, it has a discontinuity at $x=K$ and secondly the diffusion term becomes zero at $x=0$.
So we expect that the Delta equation has only a viscosity solution. For ease of
reading, we recall the notion of viscosity solution in the next subsection.

\subsection{Viscosity solution}

Consider the following general problem
\begin{equation}  \label{def_s200}
\left\{%
\begin{array}{l}
u_t+F(x,t,u,u_x,u_{xx})=0, \qquad (x,t)\in Q_{T}=\Omega\times]0,T], \\
u(x,t)=g(x,t),\qquad (x,t)\in\partial\Omega\times[0,T[ \\
u(x,0)=u_0(x),\qquad x\in\bar\Omega%
\end{array}%
\right.
\end{equation}
where $F: \mathds{R}^5 \rightarrow \mathds{R} $ is a continuous function and
$g$, $u_0$ are given functions. To extend the technique of viscosity
solutions to functions that need not to be continuous, we define
\begin{equation}  \label{h100}
{u}^{*}(x,t)=\lim_{r \to 0 } sup \{ u(y,\tau):|(y,t)-(x,t)| \leq r,\,
(y,\tau)\in Q_{T} \}
\end{equation}
and
\begin{equation}  \label{h101}
{u}_{*}(x,t)=\lim_{r \to 0 } inf \{ u(y,t):|(y,t)-(x,t)| \leq r,\,
(y,\tau)\in Q_{T} \}
\end{equation}
which are called the upper and lower semicontinuous envelopes of $u$
respectively.

\begin{definition}
\label{def_vis} Let $u: \bar{Q}_{T} \rightarrow \mathds{R}$ be a bounded
function and $\phi :\bar{Q}_{T} \rightarrow \mathds{R}$ be any $C^{2,1}(\bar{%
Q}_T)$ function.

\begin{itemize}
\item The function $u$ is called a viscosity subsolution of (\ref{def_s200})
if, whenever $(x_{0},t_0)$ is a maximum of $u^* -\phi$, the following
conditions are satisfied:
\begin{equation}  \label{def_s2000}
\phi_t(x_0,t_0)+F(x_{0},t_0,u^*(x_{0},t_0),\phi_x
(x_{0},t_0),\phi_{xx}(x_{0},t_0))\leq 0,
\end{equation}
when $(x_0,t_0)\in Q_T$,

$u^*(x_0,t_0)-g(x_0,t_0)\leq 0 \text{ or } (\ref{def_s2000}), \text{ when }
(x_0,t_0)\in\partial\Omega\times[0,T[,$\newline
and

$u^*(x_0,t_0)-u_{0*}(x_0)\leq 0 \text{ or } (\ref{def_s2000}), \text{ when }
x_0\in\bar\Omega,t_0=0.$

\item The function $u$ is called a viscosity supersolution of (\ref{def_s200}%
) if, whenever $(x_{0},t_0)$ is a minimum of $u_* -\phi$, the following
conditions are satisfied:
\begin{equation}  \label{def_s2001}
\phi_t(x_0,t_0)+F(x_{0},t_0,u_*(x_{0},t_0),\phi_x
(x_{0},t_0),\phi_{xx}(x_{0},t_0))\geq 0,
\end{equation}
when $(x_0,t_0)\in Q_T$,

$u_*(x_0,t_0)-g(x_0,t_0)\geq 0 \text{ or } (\ref{def_s2001}), \text{ when }
(x_0,t_0)\in\partial\Omega\times[0,T[$,\newline
and

$u_*(x_0,t_0)-u_{0}^*(x_0)\geq 0 \text{ or } (\ref{def_s2001}), \text{ when }
x_0\in\bar\Omega,t_0=0.$

\item A bounded function $u$ which is both a viscosity subsolution and a
supersolution is called a (non necessarily continuous) viscosity solution.
\end{itemize}
\end{definition}


\begin{nota}
\label{eq:9522} The definition of viscosity solution is an absolute one.
This means that if $u$ is a viscosity subsolution in $\Omega $, then it is
also a subsolution in $\Omega^{\prime }$, where $\Omega^{\prime }\subset
\Omega $.
\end{nota}

\begin{nota}
\label{eq:9523} In the definition of viscosity solution, local maximum can
be replaced by global maximum and also by strict local or global maximum.
Moreover $C^{2,1}$ functions can be substituted by smooth functions. Also we can
assume that the local maximum is zero. Similar remarks apply to supersolutions.
\end{nota}

To prove the existence of viscosity solutions to Problem (\ref{delta_eq})-(%
\ref{eq:4002}) and (\ref{eq_fr_b}) we will use the vanishing viscosity
method to construct a sequence of approximate solutions $u^{\varepsilon}$, we will
then apply the method of weak limits to prove the convergence of the
sequence to the viscosity solution of the Delta equation. The advantage of
the weak limit method  is that it allows passing to the limits with only uniforme $L^{\infty}$ estimates on $u^{\varepsilon}$.

Problem (\ref{delta_eq})-(\ref{eq:4002}) and (\ref{eq_fr_b}) can be
written as (\ref{def_s200}) with $\Omega =]0,b[$,
\begin{eqnarray}
F(x,t,u,u_{x},u_{xx}) &=&-\left( 0.5x^{2}\sigma ^{2}\left(
1+2e^{rt}a^{2}x^{2}u_{x}\right) \right) u_{xx}  \notag  \label{eq:15000} \\
&&-2\sigma ^{2}e^{rt}a^{2}x^{3}u_{x}^{2}-(r-q+\sigma ^{2})xu_{x}+qu
\end{eqnarray}%
with $g(0,t)=1$, $g(b,t)=1$ and $u_{0}$ defined by (\ref{eq:4001}).

\section{Regularization}

In order to overcome the difficulties, in dealing with the Delta Equation, raised above,
we propose a small perturbation of the initial problem by
considering a new approximate problem. We  seek  a function $
u^{\varepsilon}(x,t)$ that satisfies the equation
\begin{equation}  \label{epsilon_eq}
u^{\epsilon}_t=(\left(a^{\varepsilon}_0(t,x,u^{\varepsilon}_x)u^{
\varepsilon}_x\right)_{x}+(r-q)xu^{\varepsilon}_x-qu^{\varepsilon}, \quad
(x,t)\in Q_T.
\end{equation}
with $a_0^{\varepsilon}=0.5x^{2}\sigma^{2}\big(
1+e^{rt}a^{2}x^{2}u_{x}^{\varepsilon}\big)+\varepsilon$, $\varepsilon>0$ small,
 boundary conditions (\ref{eq:4002}), (\ref{eq_fr_b}) and  initial
condition
\begin{eqnarray}  \label{eq:40011}
u^{\varepsilon}(x,0)=
\begin{cases}
u_{0}(x),\qquad 0\leq x<K-\varepsilon \\
H_{5}(x),\qquad K-\varepsilon \leq x\leq K+\varepsilon \\
u_{0}(x), \qquad K+\varepsilon<x \leq b \\
\end{cases}%
\end{eqnarray}
where $H_{5}(x)$ is the Hermite polynomial of degree $5$ that satisfies $%
H_5(K-\varepsilon)=u_0(K-\varepsilon)$, $H_5(K+\varepsilon)=u_0(K+%
\varepsilon)$, $H_5^{\prime }(K-\varepsilon)=H_5^{\prime \prime
}(K-\varepsilon)=H_5^{\prime }(K+\varepsilon)=H_5^{\prime \prime
}(K+\varepsilon)=0$.\newline
In order to define  a  weak solution, we  multiply (\ref{epsilon_eq}) by $\psi$
and integrate with respect to $x$:
\begin{equation*}
\int\limits_{0}^{b} u^{\varepsilon}_{t}\psi\,dx -
\int\limits_{0}^{b}(a^{\varepsilon}_{0}u^{\varepsilon}_{x})_{x}\psi\,dx
-\int\limits_{0}^{b}(r-q)xu^{\varepsilon}_{x}\psi\,dx +
\int\limits_{0}^{b}qu^{\varepsilon}\psi\,dx =0.
\end{equation*}
Then, integrating  by parts and assuming that $\psi(0)=\psi(b)=0$, we obtain
\begin{equation}  \label{eq:10}
\int\limits_{0}^{b} u^{\varepsilon}_{t}\psi\,dx +
\int\limits_{0}^{b}a^{\varepsilon}_{0}u^{\varepsilon}_{x}\psi_{x}\,dx -
\int\limits_{0}^{b}(r-q)xu^{\varepsilon}_{x}\psi\,dx +
\int\limits_{0}^{b}qu^{\varepsilon}\psi\,dx =0.
\end{equation}
Note that for Equation (\ref{eq:10}) to make sense, we must have $u^{\varepsilon}$, $%
u^{\varepsilon}_t$ and $u^{\varepsilon}_x$ $\in L^2(0,b)$, for $t \in ]0,T]$. Taking into consideration conditions  (\ref{eq:4002}), (\ref{eq_fr_b}) we choose
the test function space to be
\begin{equation*}
V_{0} = \left\lbrace \psi, \psi_x \in L^2(0,b):\psi(0) = \psi(b)=0
\right\rbrace,
\end{equation*}
and for the space solution we consider
\begin{equation*}
V=\{u,u_{t},u_{x}\in L^{2}(0,b):u(0,t)=0,u(b,t)=1,\, \text{ for all }t\in
\lbrack 0,T]\}.
\end{equation*}

\begin{definition}[Weak solution]
\label{def_sf} A function $u^{\varepsilon}\in V$ is said to be a weak
solution of problem (\ref{epsilon_eq})-(\ref{eq:40011}) if it satisfies (\ref%
{eq:40011}) and  (\ref{eq:10}) for all $\psi \in V_0$, and $t
\in ]0,T[$. Relation (\ref{eq:10}) must be understood as an equality in $%
D^{\prime }(0,T)$.
\end{definition}

We are interested in obtaining classical solutions using the Leray-Schauder
existence theory. So following \cite{LSU68}, we need to obtain à priori
estimates for all possible weak solutions of Equation (\ref{epsilon_eq}).

\section{ À priori estimates}

This section is entirely devote  to obtaining the desired estimates. The main
tools required are integral calculations, imbedding inequalities and Gronwall's Lemmas.
Since we have the conditions
\begin{equation*}
\psi(0)=\psi(b)=0, \qquad u^{\varepsilon}(0,t)=0 \qquad \mathrm{and} \qquad
u^{\varepsilon}(b,t)=1,
\end{equation*}
we are not able to considerer $\psi=u^{\varepsilon}$ in (\ref{eq:10}), then
we will introduce a new function $v$ defined by
\begin{equation*}
v(x,t)=u^{\varepsilon}(x,t)- \frac{x}{b}.
\end{equation*}
It has the properties
\begin{equation*}
u^{\varepsilon}(x,t)=v(x,t)+ \frac{x}{b},
\end{equation*}
\begin{equation*}
u_{t}^{\varepsilon}(x,t)=v_{t}(x,t),
\end{equation*}
\begin{equation*}
u_{x}^{\varepsilon}(x,t)=v_{x}(x,t)+ \frac{1}{b}\qquad \text{and} \qquad
u_{xx}^{\varepsilon}(x,t)=v_{xx}(x,t).
\end{equation*}
According to (\ref{epsilon_eq}) the new function satisfies
\begin{equation}  \label{prob_aux}
v_t= \left(a^{\varepsilon}_{0}\left(v_{x}+\frac{1}{b}\right)\left(v_{x}+%
\frac{1}{b}\right)\right)_{x}+(r-q)x\left(v_{x}+\frac{1}{b}\right)-q\left(v+%
\frac{x}{b}\right),
\end{equation}
that is,
\begin{equation}  \label{prob2_aux}
v_t= ((a_{1}+a_{2}+\varepsilon )v_{x})_{x}+(r-q)xv_{x}-qv+f, \qquad (x,t)\in
Q_T
\end{equation}
with
\begin{equation*}
a_{1}=0.5x^{2}\sigma^{2}\left(1+ e^{rt}a^{2}x^{2}\left(v_{x}+\frac{1}{b}%
\right)\right),
\end{equation*}
\begin{equation*}
a_{2}=\frac{0.5\sigma^{2}e^{rt}a^{2}x^{4}}{b},
\end{equation*}
\begin{equation*}
f=\frac{(\sigma^{2}+r+2q)x}{b}+ \frac{2x^{3}\sigma^{2}e^{rt}a^{2}}{b^{2}}.
\end{equation*}
The boundary and initial conditions for this  new problem are
\begin{equation}  \label{cfv}
v(0,t)=0, \quad v(b,t)=0,
\end{equation}
\begin{eqnarray}  \label{eq:40014}
v(x,0)=v_{0}(x)=%
\begin{cases}
-\frac{x}{b},\qquad 0\leq x<k-\varepsilon \\
H_{5}(x)-\frac{x}{b},\qquad k-\varepsilon \leq x\leq k+\varepsilon \\
1-\frac{x}{b}, \qquad k+\varepsilon<x \leq b. \\
\end{cases}%
.
\end{eqnarray}
The definition of weak solution to Problem (\ref{prob2_aux})-(\ref{eq:40014}) is similar to Definition \ref{def_sf}.

\begin{nota}
\label{eq:95} Assumption $u_{x}\geq 0$ is equivalent to $v_{x}(x,t)+ \frac{1%
}{b} \geq 0$, and then $a_1\geq 0$, $(x,t)\in Q_T$.
\end{nota}

First we prove that the initial condition is smooth.

\begin{lemma}
Let $v_0(x)$ be as defined in (\ref{eq:40014}). Then, for $x\in[0,b]$,
\begin{eqnarray}
&&|v_0(x)|\leq C,  \label{v0linf} \\
&& |v'_0(x)|\leq C\varepsilon^{ -1},\qquad \text{ and }  \label{v0xlinf} \\
&& |v''_0(x)|\leq C\varepsilon^{ -2},  \label{v0xxlinf}
\end{eqnarray}
where $C$ does not depend on $\varepsilon$.
\end{lemma}

\begin{proof}
Regarding  (\ref{eq:40014}) we only need to prove the results for $H_5(x)$ were  $H_5(x)$ is as defined  in (\ref%
{eq:40011}). Applying the divided difference method, we can write
\begin{eqnarray*}
&&H_{5}(x)=\frac{1}{8\varepsilon^{3}}(x-K+\varepsilon)^{3} -\frac{3}{16\varepsilon^{4}}(x-K+\varepsilon)^{3}(x-K-\varepsilon)
\\&& +\frac{3}{16\varepsilon^{5}}(x-K+\varepsilon)^{3}(x-K-\varepsilon)^{2}\qquad x \in [K-\varepsilon,K+\varepsilon].
\end{eqnarray*}
Using the triangular inequality, we have
$$|H_{5}(x)|\leq\frac{1}{8\varepsilon^{3}}|x-K+\varepsilon|^{3} +\frac{3}{16\varepsilon^{4}}|x-K+\varepsilon|^{3}|x-K-\varepsilon|
 +\frac{3}{16\varepsilon^{5}}|x-K+\varepsilon|^{3}|x-K-\varepsilon|^{2}.$$
It is clear that $|x-K+\varepsilon|\leq 2\varepsilon$ and $|x-K-\varepsilon|\leq 2\varepsilon$ for $x \in [K-\varepsilon,K+\varepsilon]$, and so
$$|H_{5}(x)|\leq10,$$
which prove (\ref{v0linf}).
Differentiating $H_5(x)$, we obtain
\begin{eqnarray*}
H'_{5}(x)=\frac{3}{8\varepsilon^{3}}(x-K+\varepsilon)^{2} -\frac{9}{16\varepsilon^{4}}(x-K+\varepsilon)^{2}(x-K-\varepsilon)
-\frac{3}{16\varepsilon^{4}}(x-K+\varepsilon)^{3}&&\\ +\frac{9}{16\varepsilon^{5}}(x-K+\varepsilon)^{2}(x-K-\varepsilon)^{2} +\frac{6}{16\varepsilon^{5}}(x-K+\varepsilon)^{3}(x-K-\varepsilon)
&&\end{eqnarray*}
Using  the triangular inequality we have
$$|H'_{5}(x)|\leq\frac{45}{2}\varepsilon^{-1}$$ and so (\ref{v0xlinf}) is  proved. The same technique  permits us to prove (\ref{v0xxlinf}) by estimating
$$|H''_{5}(x)|\leq\frac{81}{2}\varepsilon^{-2}.$$

\end{proof}

\begin{theorem}
Let $v$ be a weak solution of (\ref{prob2_aux})-(\ref{eq:40014}) that
satisfies Remark \ref{eq:95}. Then
\begin{eqnarray}
&&\|v\|_{L^{\infty}(0,T;L^{2}(0,b))}\leq C  \label{ul2} \\
&\text{ and }& \|v_x\|_{L^2(0,T;L^{2}(0,b))}\leq C\varepsilon^{-\frac{1}{2}}
\label{uxl2}
\end{eqnarray}
where $C=C(a,\sigma,r,q,b,T)$ does not depend on $\varepsilon$.
\end{theorem}

\begin{proof}
Considering $\psi=v$ in the definition of a weak solution, we obtain
$$ \int\limits_{0}^{b} v_{t}v\,dx +\int\limits_{0}^{b}(a_{1}+a_{2}+ \varepsilon)v^{2}_{x}\,dx =(r-q)\int\limits_{0}^{b}x v_{x}v\,dx   -q\int\limits_{0}^{b}v^{2}\,dx
  +\int\limits_{0}^{b}fv\,dx.$$
Integrating  by parts  the first term on the right hand side we obtain
$$\frac{1}{2}\frac{d}{dt}\int\limits_{0}^{b}  v^{2}\,dx + \int\limits_{0}^{b}(a_{1}+a_{2}+ \varepsilon)v^{2}_{x}\,dx = -\frac{r+q}{2}\int\limits_{0}^{b}v^{2}\,dx-q\int\limits_{0}^{b}v^{2}\,dx
  +\int\limits_{0}^{b}fv\,dx.$$
By  Cauchy's inequality,
\begin{eqnarray*}
 \frac{1}{2}\frac{d}{dt}\|v\|^{2}_{L^{2}(0,b)}+ \varepsilon \|v_{x}\|^{2}_{L^{2}(0,b)}\leq \frac{1}{2}\|f\|^{2}_{L^{2}(0,b)}+\frac{1}{2}\|v\|^{2}_{L^{2}(0,b)}.
\end{eqnarray*}
Integrating with respect to $t$,
\begin{eqnarray*}
\|v(x,t)\|^{2}_{L^{2}(0,b)}+ 2\varepsilon \int_{0}^{t} \|v_{x}\|^{2}_{L^{2}(0,b)}\,dt\leq\int_{0}^{T}\|f\|^{2}_{L^{2}(0,b)}\,dt+\int_{0}^{T}\|v\|^{2}_{L^{2}(0,b)}\,dt&&\\ + \|v(x,0)\|^{2}_{L^{2}(0,b)}.&&
\end{eqnarray*}
Taking into account  that $v_{0} \in L^{2}(0,b)$ and $\int_{0}^{T}\|f\|^{2}_{L^{2}(0,b)}\, dt < C$ and  applying   Gronwall's inequality, we prove (\ref{ul2}), and then (\ref{uxl2}) follows easily.
\end{proof}

 The next theorem shows that $v$ is uniformly bounded in $L^{\infty}$.

\begin{theorem}
\label{teo_v_inf} Let $v$ be a weak solution of (\ref{prob2_aux})-(\ref%
{eq:40014}) that satisfies Remark \ref{eq:95}. Then
\begin{eqnarray}  \label{555}
\|v\|_{L^{\infty}(0,T,L^{\infty}(0,b))}\leq C
\end{eqnarray}
where $C=C(a,\sigma,r,q,b,T)$ does not depend on $\varepsilon$.
\end{theorem}

\begin{proof}
Multiplying  (\ref{prob2_aux}) by  $v^{2k-1}$,  $k\geq 1$, and integrating with respect to $x$, we obtain
  \begin{eqnarray*}
\int\limits_{0}^{b} v_{t}v^{2k-1}\,dx +(2k-1) \int\limits_{0}^{b}(a_{1}+a_{2}+ \varepsilon)v^{2k-2}v^{2}_{x}\,dx =(r-q)\int\limits_{0}^{b}x v_{x}v^{2k-1}\,dx &&\\ \nonumber -q\int\limits_{0}^{b}v^{2k}\,dx
  +\int\limits_{0}^{b}fv^{2k-1}\,dx.&&
  \end{eqnarray*}
 Applying integration by parts to the first term on the right-hand side, we have
  \begin{eqnarray*}
\frac{1}{2k}\frac{d}{dt}\int\limits_{0}^{b}  v^{2k}\,dx +(2k-1) \int\limits_{0}^{b}(a_{1}+a_{2}+ \varepsilon)v^{2k-2}v^{2}_{x}\,dx &&\\ \nonumber= -\frac{r+(2k-1)q}{2k}\int\limits_{0}^{b}v^{2k}\,dx
  +\int\limits_{0}^{b}fv^{2k-1}\,dx. &&
\end{eqnarray*}
By Remark \ref{eq:95}, we can ignore the second term on the left-hand side and then, applying the  H$\ddot{o}$lder's inequality, we arrive at
\begin{eqnarray*}
 \|v\|^{2k-1}_{L^{2k}(0,b)}\frac{d}{dt} \|v\|_{L^{2k}(0,b)}\leq \|f\|_{L^{2k}(0,b)}\|v\|^{2k-1}_{L^{2k}(0,b)}.
\end{eqnarray*}
Cancelling the term $\|v\|^{2k-1}_{L^{2k}(0,b)}$ we have
\begin{eqnarray*}
\frac{d}{dt} \|v\|_{L^{2k}(0,b)}\leq \|f\|_{L^{2k}(0,b)}.
\end{eqnarray*}
Integrating with respect to $t$ it results in
\begin{eqnarray*}
 \|v\|_{L^{2k}(0,b)}\leq \|v_{0}\|_{L^{2k}(0,b)} + \int_{0}^{T}\|f\|_{L^{2k}(0,b)}\,dt.
\end{eqnarray*}
Since $v_{0} \in L^{\infty}(0,b)$ and $\int_{0}^{T}\|f\|_{L^{\infty}(0,b)}\, dt < C$, taking  $k\rightarrow \infty$, (\ref{555}) follows.
\end{proof}

Equation (\ref{prob2_aux}) can be rewritten as
\begin{equation}  \label{prob3_aux}
v_t= (a_{1}+a_{3}+\varepsilon )v_{xx}+ a_{4}v^{2}_{x} +a_{5}v_{x}-qv+f
\end{equation}
with\newline
$a_{3}=0.5x^{4}\sigma^{2}e^{rt}a^{2}\left(v_{x}+ \frac{1}{b}\right)\geq 0$
by Remark \ref{eq:95}, \newline
$a_{4}=2x^{3}\sigma^{2}e^{rt}a^{2}$ and\newline
$a_{5}=\frac{4x^{3}\sigma^{2}e^{rt}a^{2}}{b} +(r-q+\sigma^{2})x$.\newline

\begin{theorem}
If  $v$ is a weak solution of (\ref{prob2_aux})-(\ref{eq:40014}) that
satisfies Remark \ref{eq:95}, then
\begin{eqnarray}
&&\|v_{x}\|_{L^{\infty}(0,T;L^{2}(0,b))}\leq C\varepsilon^{-\frac{3}{2}}
\label{uxlinf} \\
&\text{ and }& \|v_{xx}\|_{L^2(0,T;L^{2}(0,b))}\leq C\varepsilon^{-2}
\label{uxxl2}
\end{eqnarray}
where $C=C(a,\sigma,r,q,b,T)$ does not depend on $\varepsilon$.
\end{theorem}

\begin{proof}
Multiplying  (\ref{prob3_aux}) by $ v_{xx} $, and integrating with respect to $x$, we obtain
 \begin{eqnarray*}
\int\limits_{0}^{b} v_tv_{xx}\,dx=\int\limits_{0}^{b} (a_{1}+a_{3}+\varepsilon )v_{xx}^{2}\,dx+ \int\limits_{0}^{b} a_{4}v^{2}_{x}v_{xx}\,dx &&\\ \nonumber +\int\limits_{0}^{b} a_{5}v_{x}v_{xx}\,dx-q\int\limits_{0}^{b}v v_{xx}\,dx+ \int\limits_{0}^{b} f v_{xx}\,dx&&
\end{eqnarray*}
\begin{eqnarray*}
\Leftrightarrow  \big[ v_{x}v_{t}\big]^{b}_{0}-\int\limits_{0}^{b} v_xv_{xt}\,dx-\int\limits_{0}^{b} (a_{1}+a_{3}+\varepsilon )v_{xx}^{2}\,dx= \frac{1}{3}\big[ v^{3}_{x}a_{4}\big]^{b}_{0}&&\\ \nonumber-\int\limits_{0}^{b}\frac{v^{3}_{x}}{3} (a_{4})_{x}\,dx+\frac{1}{2}\big[ v^{2}_{x}a_{5}\big]^{b}_{0}-\int\limits_{0}^{b}\frac{v^{2}_{x}}{2} (a_{5})_{x}\,dx  -q\int\limits_{0}^{b}v v_{xx}\,dx+ \int\limits_{0}^{b} f v_{xx}\,dx,&&
\end{eqnarray*}
where we used  integration by parts. By  H$\ddot{o}$lder's inequality we arrive at
\begin{eqnarray*}
\frac{1}{2} \frac{d}{dt}\int\limits_{0}^{b} v_{x}^{2}\,dx +\varepsilon \int\limits_{0}^{b}v^{2}_{xx}\,dx= Cv^{3}_{x}(b,t)+ C\int\limits_{0}^{b} v^{3}_{x}\,dx+Cv^{2}_{x}(b,t)&& \\ \nonumber + C\int_{0}^{b}v^{2}_{x}\,dx+ \frac{\varepsilon}{3}\int_{0}^{b}v^{2}_{xx}\,dx +C\varepsilon^{-1} \int\limits_{0}^{b} v^{2}\,dx+ \frac{\varepsilon}{3}\int_{0}^{b}v^{2}_{xx}\,dx+C\varepsilon^{-1} \int\limits_{0}^{b} f^{2}\,dx. &&
\end{eqnarray*}
According  to (\ref{eq:463}) it is reasonable to consider $v_x(b,t)=0$, for $b$ sufficiently large. Now we use the  Gagliardo-Nirenberg interpolation inequalities to eliminate  the terms with $v^3_x$ and thus  we obtain
\begin{eqnarray*}
 \frac{1}{2} \frac{d}{dt}\int\limits_{0}^{b} v_{x}^{2}\,dx +\frac{\varepsilon}{3} \int\limits_{0}^{b}v^{2}_{xx}\,dx \leq C\varepsilon^{-1}+\frac{\varepsilon}{4}\int\limits_{0}^{b}v^{2}_{xx}\,dx  + C\varepsilon^{-3}\int\limits_{0}^{b}v^{6}\,dx.
\end{eqnarray*}
Integrating with respect to $t$ and recalling the previous theorems, the result follows.
\end{proof}

\begin{corol}
\label{uxl2inf} If $u$ is a weak solution of (\ref{epsilon_eq}) that
satisfies $u_x^{\varepsilon}>0$, then
\begin{eqnarray*}
\int_0^T\|u_x^{\varepsilon}\|_{L^{\infty}(0,b)}^{2}\,dt\leq C\varepsilon^{-4}
\end{eqnarray*}
where $C=C(a,\sigma,r,q,b,T)$ does not depend on $\varepsilon$.
\end{corol}

\begin{proof}
By the interpolation inequality and H\"older's inequality, we have
\begin{eqnarray*}
\|u_x^{\varepsilon}\|_{L^{\infty}(0,b)}&\leq &C\|u_{xx}^{\varepsilon}\|_{L^2(0,b)}^{\frac12}\|u_{x}^{\varepsilon}\|_{L^2(0,b)}^{\frac12}\\
\|u_x^{\varepsilon}\|_{L^{\infty}(0,b)}^2&\leq &C\|u_{xx}^{\varepsilon}\|_{L^2(0,b)}^{2}+C\|u_{x}^{\varepsilon}\|_{L^2(0,b)}^{2}
\end{eqnarray*}
Integrating with respect to $t$ and using the previous Theorem the result follows.
\end{proof}

\begin{theorem}
Let $v$ be a weak solution of (\ref{prob2_aux})-(\ref{eq:40014}) that
satisfies Remark \ref{eq:95}. Then
\begin{eqnarray*}
\|v_{t}\|_{L^2(0,T;L^{2}(0,b))}\leq C\varepsilon^{-1}
\end{eqnarray*}
where $C=C(a,\sigma,r,q,b,T)$ do not depends on $\varepsilon$.
\end{theorem}

\begin{proof}
Multiplying  (\ref{prob2_aux}) by $ v_{t} $, and integrating with respect to $x$, we obtain
 \begin{eqnarray*}
\int\limits_{0}^{b} v_{t}^{2}\,dx +\int\limits_{0}^{b}(a_{1}+a_{2}+ \varepsilon)v_{x}v_{xt}\,dx=(r-q)\int\limits_{0}^{b}x v_{x}v_{t}\,dx   -q\int\limits_{0}^{b}vv_{t}\,dx &&\\ \nonumber
  +\int\limits_{0}^{b}fv_{t}\,dx &&\\ \nonumber  \Leftrightarrow
\int\limits_{0}^{b}  v_{t}^{2}\,dx +\frac{1}{2} \int\limits_{0}^{b}(a_{1}+a_{2}+ \varepsilon)(v^{2}_{x})_t\,dx \leq C\int_{0}^{b}v^{2}_{x}\,dx + \frac{1}{4}\int_{0}^{b} v^{2}_{t}\,dx&& \\ \nonumber - \frac{q }{2}\frac{d}{dt}\int_{0}^{b} v^{2}\,dx -C\int_{0}^{b}f^{2}\,dx+\frac{1}{4}\int_{0}^{b}v_t^{2}\,dx.&&
\end{eqnarray*}
Concerning  the second term on the left hand side we can write
\begin{eqnarray*}
 \frac{1}{2}\int\limits_{0}^{b}(a_{1}+a_{2}+ \varepsilon)(v^{2}_{x})_t\,dx  =\frac{1}{2}\frac{d}{dt}\int\limits_{0}^{b}(a_{1}+a_{2}+ \varepsilon)v^{2}_{x}\,dx &&\\ \nonumber- \frac{1}{2}\int\limits_{0}^{b}(a_{1}+a_{2}+ \varepsilon)_tv^{2}_{x}\,dx.&&
\end{eqnarray*}
Since
\begin{eqnarray*}
&&(a_{2})_t=\frac{0.5rx^{4}\sigma^{2}e^{rt}a^{2}}{b}=ra_2,\\&&
(a_{1})_t= 0.5rx^{4}\sigma^{2}e^{rt}a^{2}v_{x}+ 0.5x^{4}\sigma^{2}e^{rt}a^{2}v_{xt}+ \frac{0.5rx^{4}\sigma^{2}e^{rt}a^{2}}{b}\\
&&\phantom{\frac{d a_{1}}{dt}}=rba_2v_x+ba_2v_{xt}+ra_2,
\end{eqnarray*}
we have
\begin{eqnarray*}
\frac{1}{2}\int\limits_{0}^{b}  v_{t}^{2}\,dx +\frac{1}{2}\frac{d}{dt} \int\limits_{0}^{b}(a_{1}+ a_{2}+ \varepsilon)v^{2}_{x}\,dx + \frac{q}{2}\frac{d}{dt}\int_{0}^{b}v^{2}\,dx\leq \frac{1}{2}\int_{0}^{b}bra_{2}v^{3}_{x}\,dx &&\\ \nonumber + \int_{0}^{b}ra_{2}v^{2}_{x}\,dx + \frac{1}{2}\int_{0}^{b}ba_{2}v_{xt}v_{x}^{2}\,dx+ C\int_{0}^{b}v^{2}_{x}\,dx  + C\int_{0}^{b}f^{2}\,dx. &&
 \end{eqnarray*}
Taking into account that
 \begin{eqnarray*}
  \frac{1}{2}\int_{0}^{b}ba_{2}v_{xt}v_{x}^{2}\,dx =\frac{1}{6}\int_{0}^{b}ba_{2}(v_{x}^{3})_{t}\,dx=\frac{1}{6} \frac{d}{dt}\int_{0}^{b}ba_{2}v^{3}_{x}\,dx-\frac{1}{6} \int_{0}^{b}bra_{2}v^{3}_{x}\,dx,
\end{eqnarray*}
we arrive at
 \begin{eqnarray*}
\frac12\int\limits_{0}^{b}  v_{t}^{2}\,dx +\frac{1}{2}\frac{d}{dt} \int\limits_{0}^{b}(a_{1}+ a_{2}- \frac{1}{3}ba_{2}v_x+ \varepsilon)v^{2}_{x}\,dx +\frac{q}{2}\frac{d}{dt}\int_{0}^{b}v^{2}\,dx \leq &&\\ \nonumber \frac{1}{3}\int_{0}^{b}bra_{2}v^{3}_{x}\,dx +\int_{0}^{b}ra_{2}v^{2}_{x}\,dx + C\int_{0}^{b}v^{2}_{x}\,dx + C\int_{0}^{b}f^{2}\,dx.&&
 \end{eqnarray*}
Define $a_{6}=0.5x^{2}\sigma^{2}(1+x^{2}a^{2}e^{rt}(\frac{2}{3}v_{x}+\frac{1}{b}))\geq 0$ by Remark \ref{eq:95}. If we apply  the Gagliardo-Nirenberg inequalities, the last equation becomes
\begin{eqnarray*}
\frac12\int\limits_{0}^{b}  v_{t}^{2}\,dx +\frac{1}{2}\frac{d}{dt} \int\limits_{0}^{b}(a_{6}+ a_{2}+ \varepsilon)v^{2}_{x}\,dx +\frac{q}{2}\frac{d}{dt}\int_{0}^{b}v^{2}\,dx \leq C\int_{0}^{b}v_{xx}^2\,dx  &&\\ \nonumber +C\int_{0}^{b}v^{6}\,dx + C\int_{0}^{b}v^{2}_{x}\,dx + C\int_{0}^{b}f^{2}\,dx.&&
 \end{eqnarray*}
 Integrating with respect to $t$ we have
\begin{eqnarray*}
 \int\limits_{0}^{t}  \int\limits_{0}^{b}  v_{t}^{2}\,dx \,dt+ \int\limits_{0}^{b}(a_{2}+a_{6} + \varepsilon)v^{2}_{x}\,dx +q\int_{0}^{b}v^{2}\,dx\leq C \int\limits_{0}^{T} \int_{0}^{b}v^{2}_{xx}\,dx\,dt &&\\ \nonumber + C\int\limits_{0}^{T} \int_{0}^{b}v^{6}\,dx\,dt +C\int\limits_{0}^{T}\int_{0}^{b}f^{2}\,dx\,dt + C\int\limits_{0}^{T} \int_{0}^{b}v^{2}\,dx\,dt+ q\int_{0}^{b}v^{2}_0\,dx&&\\\nonumber+\int\limits_{0}^{b}(a_{2}+a_{6} + \varepsilon)v^{2}_{x}(x,0)\,dx. &&
\end{eqnarray*}
Applying the previous theorem the required  result follows.
\end{proof}

Writing equation (\ref{epsilon_eq}) as
\begin{eqnarray}
u^{\varepsilon}_t=\left(0.5x^{2}\sigma^{2}+\sigma^{2}e^{rt}a^{2}x^{4}u^{%
\varepsilon}_{x}+\varepsilon\right)u^{\varepsilon}_{xx}
+2\sigma^{2}a^{2}e^{rt}x^{3}(u^{\varepsilon}_{x})^2 && \\
+(r-q+\sigma^{2})xu^{\varepsilon}_x-qu^{\varepsilon},&&  \notag
\label{eq_eps_nd}
\end{eqnarray}
and differentiating with respect to $x$, we obtain
\begin{eqnarray*}
u^{\varepsilon}_{tx}=\left(\left(0.5x^{2}\sigma^{2}+%
\sigma^{2}e^{rt}a^{2}x^{4}u^{\varepsilon}_{x}
+\varepsilon\right)u^{\varepsilon}_{xx}\right)_{x}
+4\sigma^{2}e^{rt}a^{2}x^{3}u^{\varepsilon}_{x}u^{\varepsilon}_{xx}&& \\
+(r-q+\sigma^2)xu^{\varepsilon}_{xx}+6\sigma^{2}e^{rt}a^{2}x^{2}(u^{%
\varepsilon}_{x})^{2}+(r-2q+\sigma^2)u^{\varepsilon}_{x}. &&  \notag
\end{eqnarray*}
Setting $w=u^{\varepsilon}_{x}$, we arrive at the equation
\begin{eqnarray}  \label{prob51_aux}
w_{t}=\left(\left(0.5x^{2}\sigma^{2}+\sigma^{2}e^{rt}a^{2}x^{4}w
+\varepsilon\right)w_{x}\right)_{x}+4\sigma^{2}e^{rt}a^{2}x^{3}ww_{x} &&
\notag \\
+(r-q+\sigma^{2})xw_{x}+6\sigma^{2}e^{rt}a^{2}x^{2}w^{2}
+(\sigma^{2}-2q+r)w&&
\end{eqnarray}
with initial and  boundary conditions
\begin{equation}  \label{prob51_aux_cond}
w(x,0)=v_0^{\prime }(x),\quad w(0,t)= w(b,t)=0.
\end{equation}

\begin{theorem}
Let $w$ be the weak solution of (\ref{prob51_aux})-(\ref{prob51_aux_cond})
then
\begin{eqnarray}  \label{wlinf}
\|w\|_{L^\infty(0,T,L^{\infty}(0,b))}\leq C\varepsilon^{-4}
\end{eqnarray}
where $C=C(a,\sigma,r,q,b,T)$ does not depend on $\varepsilon$.
\end{theorem}

\begin{proof}
The proof is similar to  that of Theorem \ref{teo_v_inf}. Multiplying  (\ref{prob51_aux}) by  $w^{2k-1}$,  $k\geq 1$, and integrating with respect to $x$, we obtain

\begin{eqnarray*}
\int\limits_{0}^{b} w_{t}w^{2k-1}\,dx +(2k-1)\int\limits_{0}^{b}\left(0.5x^{2}\sigma^{2}+\sigma^{2}e^{rt}a^{2}x^{4}w
+\varepsilon\right)w^{2}_{x}w^{2k-2}\,dx
&&\\ \nonumber=
\int\limits_{0}^{b}4\sigma^{2}e^{rt}a^{2}x^{3}w_{x}w^{2k}\,dx +\int\limits_{0}^{b}(r-q+\sigma^{2})xw_{x}w^{2k-1}\,dx&&\\ \nonumber +\int\limits_{0}^{b}6\sigma^{2}e^{rt}a^{2}x^{2}w^{2k+1}\,dx +\int\limits_{0}^{b}(\sigma^{2}+r-2q)w^{2k}\,dx.&&
\end{eqnarray*}
Applying  integration by parts to the right hand side, we get
\begin{eqnarray*}
\frac{1}{2k}\frac{d}{dt}\int\limits_{0}^{b} w^{2k}\,dx +(2k-1)\int\limits_{0}^{b}\left(0.5x^{2}\sigma^{2}+\sigma^{2}e^{rt}a^{2}x^{4}w
+\varepsilon\right)w^{2}_{x}w^{2k-2}\,dx && \\ = -\int\limits_{0}^{b}\frac{12}{2k+1}\sigma^{2}e^{rt}a^{2}x^{2}w^{2k+1}\,dx -\int\limits_{0}^{b}\frac{r-q+\sigma^{2}}{2k}w^{2k}\,dx
&&\\ \nonumber +\int\limits_{0}^{b}6\sigma^{2}e^{rt}a^{2}x^{2}w^{2k+1}\,dx +\int\limits_{0}^{b}(\sigma^{2}+r-2q)w^{2k}\,dx.&&
  \end{eqnarray*}
Taking into account the bounds of the coefficients, we have
$$
\frac{1}{2k}\frac{d}{dt}\int\limits_{0}^{b} w^{2k}\,dx +(2k-1)\varepsilon\int\limits_{0}^{b}w^{2}_{x}w^{2k-2}\,dx\leq C\int\limits_{0}^{b}w^{2k+1}\,dx +C\int\limits_{0}^{b}w^{2k}\,dx.$$
Ignoring the second term, because it is nonnegative by Remark \ref{eq:95}, and using a Sobolev embedding, we obtain
$$\|w\|_{L^{2k}}^{2k-1}\frac{d}{dt}\|w\|_{L^{2k}}  \leq C \|w\|_{L^{\infty}}^{2}b^{\frac{2k-1}{(2k-1)(2k)}}\|w\|_{L^{2k}}^{2k-1}+C\|w\|_{L^{2k}}^{2k}.$$
Therefor
$$\frac{d}{dt}\|w\|_{L^{2k}}  \leq C \|w\|_{L^{\infty}}^{2}b^{\frac{1}{2k}}+C\|w\|_{L^{2k}}.$$
Applying Gronwall's inequality and  Corollary \ref{uxl2inf},
\begin{eqnarray*}
\|w\|_{L^{2k}} &\leq& C\|w(x,0)\|_{L^{2k}}  +Cb^{\frac{1}{2k}}\int_0^T\|w\|_{L^{\infty}}^{2}\,dt\\
& \leq & C\|w(x,0)\|_{L^{2k}}  +Cb^{\frac{1}{2k}}\varepsilon^{-4}.\\
  \end{eqnarray*}
Finally taking $k\to\infty$ and attending to (\ref{v0xlinf}), the required  follows.
\end{proof}This result finishes this section.

\section{Convergence}

In this section we prove the convergence of the approximate solutions $%
u^{\varepsilon}$ of the viscosity solution $u$ of (\ref{delta_eq})-(\ref%
{eq:4002}) and (\ref{eq_fr_b}). The idea is that the so-called upper weak
limit $\overline{u}$ and the lower weak limit $\underline{u}$ are,
respectively, a viscosity subsolution and supersolution of (\ref{def_s200}). On, the one
hand, we always have $\underline{u}\leq \overline{u} $ in $Q_{T}$, and  on the
other hand, the comparison principle implies that $\overline{u}\leq
\underline{u} $ a.e. in $Q_{T}$. Finally, it is easy to see that this
equality implies the local $L^{\infty}$ convergence of $u^{\varepsilon}$ to
the function $u=\underline{u}= \overline{u} $ as $\varepsilon\rightarrow 0$,
which turns out to be a unique bounded viscosity solution of (\ref{eq:15000}).

With the estimates in the last section in hand we can now prove the existence and
uniqueness of a classical solution $u^{\varepsilon}$.

\begin{theorem}
Problem (\ref{epsilon_eq})-(\ref{eq:40011}) has a unique classical solution
\begin{equation}  \label{uc2}
u^{\varepsilon}\in C^{2+\alpha,1+\alpha/2}(\bar{Q}_T).
\end{equation}
\end{theorem}

\begin{proof}
Attending to (\ref{555}) and  (\ref{wlinf}) we are able to consider Equation (\ref{epsilon_eq}) as a linear equation with bounded coefficients. So, applying Theorem 5.1 from chapter VI in \cite{LSU68}, we obtain that $ u^{\varepsilon}_x\in C^{\alpha,\alpha/2}(\bar{Q}_T)$ $0<\alpha<1$. Which allows us to prove that the coefficients are H\"older continuous. Using the Leray-Schauder theory, namely Theorem 5.2 from chapter IV in \cite{LSU68}, we conclude that, there exists a unique classical solution of (\ref{epsilon_eq}) satisfying (\ref{uc2}).
\end{proof}

Taking into account  conditions (\ref{h100}) and (\ref{h101}), we have
\begin{eqnarray}
\overline{u}(x,t) &=& \lim_{r \to 0 } sup \{
u^{\varepsilon}(y,\tau):\|(y,\tau)-(x,t)\|\leq r, \, \varepsilon \leq r ,\,
(y,\tau) \in Q_{T}\}\quad  \label{u_up} \\
&=&\lim_{\varepsilon \to 0 } sup^{*} u^{\varepsilon}(t,x) <\infty  \notag
\end{eqnarray}
\begin{eqnarray}
\underline{u}(x,t) &= &\lim_{r \to 0 } inf \{
u^{\varepsilon}(y,\tau):\|(y,\tau)-(x,t)\|\geq r, \, \varepsilon \leq r, \,
(y,\tau) \in Q_{T}\}\quad  \label{u_down} \\
&=&\lim_{\varepsilon \to 0 } inf_{*} u^{\varepsilon}(t,x)>-\infty  \notag
\end{eqnarray}
Since $u^{\varepsilon}$ is the classical solution of (\ref{epsilon_eq})-(\ref%
{eq:40011}), it is also viscosity solution for each $\varepsilon>0$.
Applying Proposition 4.7 in Chapter V of \cite{BC08} with appropriate
adaptations to parabolic equations we can prove the next result.

\begin{theorem}
\label{teo_sss} Let  $u^{\varepsilon}(x,t)$ be the classical solution of (\ref%
{epsilon_eq})-(\ref{eq:40011}) then $\overline{u}$ and $\underline{u}$
defined by (\ref{u_up}) and (\ref{u_down}) are sub and super
viscous solutions of (\ref{eq:15000}), respectively.
\end{theorem}

In order to apply the comparison theorem, we require the next two  lemmas.

\begin{lemma}
The operator $F$ defined in (\ref{eq:15000}) is proper, that is
\begin{equation}
F(x,t,s,p,X)\leq F(x,t,s,p,Y) \qquad \text{whenever} \qquad X\geq Y
\end{equation}
and
\begin{equation}
F(x,t,r,p,X)\leq F(x,t,s,p,X) \qquad \text{whenever} \qquad r\leq s
\end{equation}
where $x,t,r,s,p,X,Y \in \mathds{R}$.\newline
\end{lemma}

\begin{proof}
By the definition of $F$ and the assumptions on the problem, we have
\begin{eqnarray*}
F(x,t,s,p,X)- F(x,t,s,p,Y)= \underbrace {\left({0.5x^{2}\sigma^{2}}(1+ 2e^{rt}a^{2}x^{2}p\right)}\limits_{\geq 0} (Y-X)\leq  0&&\\
 \quad \text{whenever} \quad X\geq Y
\end{eqnarray*}
and
\begin{eqnarray*}
F(x,t,r,p,X)- F(x,t,s,p,X)= q(r-s)<0 \qquad \text{whenever} \qquad r\leq s.
\end{eqnarray*}
\end{proof}

\begin{lemma}
\label{6501} Let $F$ be defined by (\ref{eq:15000}). Then there exists $%
\gamma>0$ such that
\begin{equation*}
\gamma(r-s)\leq F(x,t,r,p,X)- F(x,t,s,p,X) \text{ for } r\geq s \text{ and }
x,t,p,X \in \mathds{R},
\end{equation*}
and there is a function $\omega:[0,\infty]\to [0,\infty]$ that satisfies $%
\omega(0^+)=0$ such that
\begin{equation*}
F(y,t,s,\alpha(x-y),Y)-F(x,t,s,\alpha(x-y),X)\leq \omega(\alpha|x-y|^2+|x-y|)
\end{equation*}
for $x,y\in]0,b[$, $t\in]0,T[$ fixed, $X,Y\in\mathds{R}$ and $\alpha$
given by
\begin{equation}  \label{cond310}
-3\alpha\left[%
\begin{array}{cc}
1 & 0 \\
0 & 1%
\end{array}%
\right]\leq \left[%
\begin{array}{cc}
X & 0 \\
0 & -Y%
\end{array}%
\right]\leq 3\alpha \left[%
\begin{array}{cc}
1 & -1 \\
-1 & 1%
\end{array}%
\right].
\end{equation}
\end{lemma}

\begin{proof}
The first condition was proved in the previous lemma with $\gamma=q$.
with respect to the second condition, we have
\begin{eqnarray*}
&F (y,t , s, \alpha (x-y),Y)- F(x,t , s, \alpha (x-y),X)=
0.5\sigma^{2}(x^{2}X-y^{2}Y)+&\\
&+\sigma^{2}e^{rt}a^{2}\alpha (x-y)(x^{4}X- y^{4}Y)+ (r-q)\alpha(x-y)^{2} +&\\
 &+2\sigma^{2}e^{rt}a^{2}\alpha^{2} (x-y)^{2}(x^{3}-y^{3}).&
\end{eqnarray*}
Using the estimates given by (\ref{cond310}) and some power inequalities, we have
\begin{eqnarray*}
F (y,t , s, \alpha (x-y),Y)- F(x,t , s, \alpha (x-y),X)\leq C\left(|x-y| + \alpha |x-y|^{2}\right).
\end{eqnarray*}
The result now  follows, taking $\omega(s)=Cs$.
\end{proof}

Taking into account  Theorem \ref{teo_sss} and Theorem 8.2 in \cite{Crandall1992} we now have the following theorem.

\begin{theorem}
Let $\overline{u}$ and $\underline{u}$ be defined by (\ref{u_up}) and (\ref%
{u_down}) respectively then $\overline{u}\leq \underline{u}$ a.e in ${Q}_T$.
\end{theorem}

By definition $\underline{u}\leq \overline{u}$ in $Q_{T}$ and by the
previous theorem $\underline{u}\geq \overline{u}$ a.e. in $Q_{T}$ let us
consider $u=\underline{u}=\overline{u}$. Using the parabolic analogue of
Lemma 1.9 in chapter V of \cite{BC08} we conclude that, when $\varepsilon \rightarrow 0^{+}$, $u^{\varepsilon }
$ converges in $L^{\infty }$ to the viscosity solution $u=\underline{u}=\overline{u}$ of (\ref{eq:15000}).

\section{Final comments}
In this paper we  analyzed a nonlinear generalization of the Black- Scholes equations that arises when options are priced
under variable transaction costs for buying and selling underlying assets.
The mathematical model is represented by a fully nonlinear parabolic
equation with the diffusion coefficient depending linearly on the second
derivative of the option price. We proved the existence of not necessarily
continuous viscosity solutions. The vanishing viscosity method used provides a
way to determine   approximate numerical solutions. The à priori estimates
obtained are useful when we wish to prove analytically the convergence and
 convergence order of certain numerical methods. Indeed this is part of our
future work.

\section*{Acknowledgements}

This work was partially supported by the research projects: Grant N.\linebreak UID/MAT/00212/2019 - financed by FEDER through the - Programa Operacional
Factores de Competitividade, FCT - Funda\c{c}\~{a}o para a Ci\^{e}ncia e a
Tecnologia and Grant BID/ICI-FC/Santander Universidades-UBI/2016.




\end{document}